\documentclass[12pt,oneside,reqno]{amsart}
\usepackage[all]{xy}
\usepackage{amsfonts,amsmath,oldgerm,amssymb,amscd}
\UseComputerModernTips
\numberwithin{equation}{section}
\usepackage[breaklinks]{hyperref}

\usepackage{caption} 
\newtheorem{theorem}{Theorem}
\newtheorem{proposition}[theorem]{Proposition}
\newtheorem{conj}{Conjecture}
\newtheorem{lemma}[subsection]{{\bf Lemma}}

\newtheorem{remark}[subsection]{Remark}

\begin{document}
\title{Perfect powers in an alternating sum of consecutive cubes} 

\author[Das]{Pranabesh Das}
\address{Pranabesh Das, Stat-Math Unit, India Statistical Institute\\
7, S. J. S. Sansanwal Marg, New Delhi, 110016, India}
\email{pranabesh.math@gmail.com}

\author[Dey]{Pallab Kanti Dey}
\address{Pallab Kanti Dey, Stat-Math Unit, India Statistical Institute\\
7, S. J. S. Sansanwal Marg, New Delhi, 110016, India}
\email{pallabkantidey@gmail.com}

\author[Maji]{B. Maji}
\address{Bibekananda Maji, Departament of Mathematics\\ Harish-Chandra Research Institute\\ Chhatnag Road, Jhunsi\\ India, 211019}
\email{bibek10iitb@gmail.com}

\author[Rout]{S. S. Rout}
\address{Sudhansu Sekhar Rout, Departament of Mathematics\\ Harish-Chandra Research Institute\\ Chhatnag Road, Jhunsi\\ India, 211019}
\email{lbs.sudhansu@gmail.com}

\thanks{2010 Mathematics Subject Classification: Primary 11D61, Secondary 11D41, 11F11, 11F80.  \\
Keywords: Diophantine equation, Galois representation, Frey curve, modularity, level lowering, linear forms in logarithms}
\maketitle
\pagenumbering{arabic}
\pagestyle{headings}

\begin{abstract}
In this paper, we consider the problem about finding out perfect powers in an alternating sum of consecutive cubes. More precisely, we completely solve the Diophantine equation $(x+1)^3 - (x+2)^3 + \cdots - (x + 2d)^3  + (x + 2d + 1)^3 = z^p$, where $p$ is prime and $x,d,z$ are integers with $1 \leq d \leq 50$. 
 
\end{abstract}

\section{Introduction}
In 1964, Leveque \cite{WJL} proved that, if $f(x) \in \mathbb{Z}[x]$ is a polynomial of degree $k \geq 2$ with at least two simple roots, and $ n \geq \mbox{ max }\{2,5-k \}$ is an integer, then the superelliptic equation
\begin{equation}\label{eq00}
f(x) = z^n
\end{equation}
has at most finitely many solutions in integers $x$ and $z$.
This result was extended by Schinzel and Tijdeman \cite{ST}, through application of lower bounds for linear forms
in logarithms, to show that equation \eqref{eq00} has in fact at most finitely many solutions in integers $x, z$
and variable $n \geq \mbox{ max }\{2,5-k \}$. 

Earlier in 1875, Lucas \cite{EL} considered the Diophantine equation 
\begin{equation}\label{eq01}
1^2 + 2^2 +\cdots + x^2 =y^2,
\end{equation}
and asked whether the equation \eqref{eq01} has solutions in positive integers $(x,y)$ other than $(1,1)$ and $(24,70)$. Watson \cite{GNW}  completely solved the equation \eqref{eq01} and showed that there are no other solutions. 

 In 1956, Sch\"{a}ffer \cite{JS}  studied the more general equation 
\begin{equation}\label{eq02}
1^k + 2^k +\cdots + x^k =y^n.
\end{equation}
It is easy to see that for every $k$ and $n$, $(x,y)=(1,1)$ is a solution of \eqref{eq02}. Sch\"{a}ffer \cite{JS} proved that if $k\geq 1$ and $n\geq 2$ are fixed, then \eqref{eq02} has only finitely many solutions except the following cases
\begin{equation}\label{eq03}
(k,n) \in \{(1,2), (3,2), (3,4), (5,2)\}
\end{equation}
where, in each case, there are infinitely many such solutions.
In the same paper Sch\"{a}ffer gave a conjecture regarding the integral solutions of \eqref{eq02}. He conjectured that, for $k\geq 1$ and $n\geq 2$ with $(k,n)$ not in the set \eqref{eq03}, equation \eqref{eq02} has only one non-trivial solution, namely $(k,n,x,y) = (2,2,24,70)$. There are some results, at least in principle, to determine all
solutions of \eqref{eq02}. However, the bounds provided by these results are not given explicitly. Jacobson, Pint\'{e}r, Walsh \cite{JPW} confirm the conjecture for $n=2$ and $k$ even  with $k\leq 58$. Recently, Bennett, Gy\H{o}ry, Pint\'{e}r \cite{BGP} proved completely the Sch\"{a}ffer conjecture for arbitrary $n$ and $k\leq 11$. Following and extending the approach of \cite{BGP} and using modern techniques of Diophantine analysis including Baker's theory, Frey curves and the theory of modular forms, Pint\'{e}r \cite{AP} proved Sch\"{a}ffer conjecture for odd values of $k$ with $1 \leq k\leq 170$ and even values of $n$. 

Zhang and Bai \cite{ZB} generalized the equation \eqref{eq02} and considered the more general equation
\begin{equation}\label{eq04}
(x+1)^{k}+(x+2)^{k}+\cdots+(x+d)^{k}=y^{n}.
\end{equation}
They completely solved the equation \eqref{eq04} for $k=2$ and $d=x$. For $k=2$, they also proved that for a prime $p \equiv \pm 5\ (\mathrm{mod}\ 12)$ with $p\mid d$ and $\nu_{p}(d)\not \equiv 0\ (\mathrm{mod}\ n)$, the equation \eqref{eq04} has no integer solutions. Cassels \cite{SC} solved the equation \eqref{eq04} completely for $n = 2, d = 3$ and $k = 3$. Zhang \cite{ZZ} determined the perfect powers in sum of three consecutive cubes by rewritting the equation \eqref{eq04} for $k=d=3$ as 
\begin{equation}\label{eq05}
(x-1)^{3}+x^{3}+(x-1)^{3}=y^{n}.
\end{equation}
Stroeker \cite{RJS} completely solved the equation \eqref{eq04} for $k=3, n=2$ and $2\leq d\leq 50$ using linear forms in elliptic logarithms. Recently, Bennett, Patel and Siksek \cite{BPS} extended the result of Stroeker for $n \geq 3$.

Several generalizations  of \eqref{eq02} have been considered by different authors. For example Dilcher \cite{KD} studied the equation 
\begin{equation}\label{eq03a}
\chi(1)1^{k}+\chi(2)2^{k}+\cdots+\chi(xf)(xf)^{k} = by^{n}
\end{equation}
where $\chi$ is a primitive quadratic residue class character with conductor $f= f_{\chi}$ and $k, b \neq 0$ are fixed integers. This may be viewed as a {\it character-twisted} analogue of a classic  equation  of  Sch\"{a}ffer. Recently, Bennett \cite{MB} completely solved the Diophantine equation 
\begin{equation}\label{eq03a}
1^{k}- 3^{k} + 5^k - \cdots+(4x - 3)^{k} - (4x - 1)^{k} = -y^{n}
\end{equation} 
for $3 \leq k \leq 6$.


In this paper we consider the following Diophantine equation
\begin{equation}\label{equn2}
(x+1)^3 - (x+2)^3 + \cdots  +(-1)^{r-1} (x+r)^3 = z^p,
\end{equation}
where $r,x,z$ are integers and $p$ is any prime number. 
Now, for odd $r$  \eqref{equn2} reduces to the following equation 
\begin{equation}\label{equn3}
\left(x + \frac{r+1}{2} \right) \left\{ {\left(x + \frac{r+1}{2} \right)}^2 + 3 \frac{r^2 -1}{4} \right\} = z^p.
\end{equation}

Putting $r = 2 d + 1$, we have
\begin{equation}\label{equn4}
( x + d +1 ) \left\{ (x + d+ 1)^2 + 3 d (d +1) \right\} = z^p.
\end{equation}
From the  equation \eqref{equn4}, we can see that $\gcd((x +d +1), (x+d+1)^2 + 3 d(d +1))$ divides $3d(d+1)$. Hence 
\begin{equation}\label{equn7}
x+d+1 = \alpha {z_1}^p \quad \textrm{and} \quad (x+d+1)^2 + 3 d (d+1)=\beta {z_2}^p 
\end{equation}
 for some integers $z_1, z_2$ and rationals $\alpha, \beta$ with $\alpha \beta = 1$ and $z_1 z_2 = z$. The denominator and the numerator of $\alpha$  is composed of prime divisors of $3d (d+1)$. From \eqref{equn4} and \eqref{equn7}, we deduce the following ternary equation
 \begin{equation}\label{equn8}
 \beta {z_2}^p-\alpha^{2}z_{1}^{2p} = 3d(d+1).
 \end{equation}
 If $\beta < 0$, then from the equation \eqref{equn7}, we have $z_2 < 0$. Also $\alpha < 0$ as $\alpha \beta = 1$. Hence, $(\pm z_1, z_2)$ is an integral solution of equation \eqref{equn8} corresponding to $(\alpha, \beta)$ if and only if $(\pm z_1, -z_2)$ is an integral solution of equation \eqref{equn8} corresponding to $(-\alpha, -\beta)$. Therefore it is enough to solve the equation \eqref{equn8} for $\beta > 0$.
 
 \smallskip
 Suppose $S_{d}$ is the set of such pairs of positive rationals $(\alpha, \beta)$.
We need to solve the equation \eqref{equn8} for each $(\alpha, \beta)\in S_{d}$ with $1\leq d\leq 50$. Clearing denominators we can rewrite the equation \eqref{equn8} as 
\begin{equation}\label{equn9}
 r{z_2}^p-s z_{1}^{2p} = t,
\end{equation}
where $r,s,t$ are positive integers and $\gcd(r,s,t)=1$.

Now we state our main theorem as follows.
\begin{theorem}\label{thm1}
Let $r = 2 d + 1$ with $ 1\leq d \leq 50$ and let $p$ be a prime. Then the integral solutions to the equation \eqref{equn2} are given in the Table \ref{tab1}.
\end{theorem}
\smallskip
\begin{remark}
If $z = 0$, then from the equation \eqref{equn4}, we have $x = -(d+1)$ as $(x + d+ 1)^2 + 3 d (d +1) > 0$ for any $d$. Therefore, $(x,z,p) = (-d-1,0,p)$ are the trivial solutions of the equation \eqref{equn2} for any $d$.
\end{remark}
We follow the methods developed in \cite{BPS} for the proof of Theorem \ref{thm1}. We would like to point out that the main techniques used in this paper are not original and nowadays well documented in the literature. The main focus of this paper is to highlight the fact that combinations of these techniques sometimes become very handy in solving exponential diophantine equations explicitly.

\section{Perliminaries}

We use well known tools such as linear forms in two logarithms, variation of Krauss crieterion, modular method, local solubility, descent for the proof of Theorem \eqref{thm1}. In this section we provide the necessary details for these methods.

\subsection{Linear forms in $2$ logarithms:}
\smallskip

We state a special case of the following well known result of Laurent \cite{ML}.

\begin{proposition} [\cite{ML},Corollary 2)]\label{Linearform}
Let $\alpha_1$ and $ \alpha_2$ be two positive real, multiplicatively independent algebraic numbers and $\log\alpha_1, \log\alpha_2$ be any fixed determinations of the logarithms that are real and positive. Write $D = [\mathbb{Q}(\alpha_1, \alpha_2) : \mathbb{Q}]$ and 
$$ b'= \frac{b_1}{D \log A_2}  + \frac{b_2}{D \log A_1} $$
where $b_1, b_2$ are positive integers and $A_1, A_2$ are real numbers greater than one such that 
$$ \log A_i \geq \max \left\{ h(\alpha_i) , \frac{|\log \alpha_i|}{D}, \frac{1}{D}  \right\}, \quad i= 1, 2. $$
with
$$ 
h(\alpha) = \frac{1}{d} \left( \log |a| + \sum_{i=1}^d \log \max ( 1, |\alpha^{(i)}| ) \right),
$$
where $a$ is the leading coefficient of the minimal polynomial of $\alpha$ and the $\alpha^{(i)}$'s are the conjugates of $\alpha$ in $\mathbb{C}$.

Let $\Lambda = b_2 \log \alpha_2 - b_1 \log \alpha_1.$ Then 
$$ \log |\Lambda| \geq -25.2 D^4 (\max \{ \log b' + 0.38, 10/D, 1  \})^2 \log A_1 \log A_2. $$
 
\end{proposition}

\subsection{Variation of Krauss Criterion}
Now we state the following variation of Krauss criterion for the non-existence of integral solutions to the equation \eqref{equn9} for given $r,s,t$ and $p$.
\begin{lemma}[\cite{BPS}, Lemma 6.1]\label{lemma4}
Let $p \geq 3$ be prime. Let $r,s$ and $t$ be positive integers satisfying $\gcd(r,s,t)=1$. Also let $q=2kp+1$ be a prime that does not divide $r$. Define 
\begin{equation}
\mu(p,q)=\{\eta^{2p}:\eta \in \mathbb{F}_{q}\}=\{0\}\cup \{\zeta \in \mathbb{F}_{q}^{*}:\zeta^{k}=1\} 
\end{equation}
and 
\begin{equation}
B(p,q)=\{\zeta \in \mu(p,q):((s\zeta+t)/r)^{2k}\in \{0,1\}\}.
\end{equation}
If $B(p,q)=\phi$, then the equation \eqref{equn9} does not have any integral solution.
\end{lemma}

\subsection{Modular method}
Before going to our problem we would like to give a brief description about modular method. Let $E$ be an elliptic curve over $\mathbb{Q}$ of conductor $N$ and $\#E(\mathbb{F}_{q})$ be the number of points on $E$ over the finite field $\mathbb{F}_{q}$ for a good prime $q$. Let $a_{q}(E)=q+1-\#E(\mathbb{F}_{q})$. By a newform $f$ of level $N$, we mean a normalizd cusp form of weight $2$ for the full modular group. Write
$f=q+\sum_{i\geq 2}c_iq^{i}$. Write $K=\mathbb{Q}(c_1,c_2,\cdots)$ for the
totally real number field generated by the Fourier coefficients of $f$. 
\par 
We say that the curve $E$ arises modulo $p$ from the newform $f$ (and write $E\sim_{p} f$) if there is a prime ideal $\mathfrak{p}$ of $K$ above $p$ such that for all but finitely many primes $q$, we have $a_{q}(E)\equiv c_q\ (\textrm{mod}\ \mathfrak{p})$. If $f$ is a rational newform, then $f$ corresponds to some elliptic curve $F$(say). If $E$ arises modulo $p$ from $f$, then we also say that $E$ arises modulo $p$ from $F$. In this regard we have the following result.

\begin{proposition}[\cite{HC}]\label{prop3}
Let $E$ and $F$ be elliptic curves over $\mathbb{Q}$ with conductors $N$ and $N'$ respectively. Suppose that $E$ arises modulo $p$ from $F$. For all primes $q$
\begin{enumerate}
\item if $q\nmid NN'$, then $a_{q}(E)\equiv a_{q}(F)\ (\textrm{mod}\ p)$ and 
\item if $q\nmid N'$ and $q\| N$, then $q+1 \equiv \pm a_{q}(F)\ (\textrm{mod}\ p).$
\end{enumerate}
\end{proposition}

The following result provides a bound for the exponent $p$. 

\begin{proposition}[\cite{SS}]\label{prop4}
Let $E/\mathbb{Q}$ be an elliptic curve of conductor $N$ with $t \mid \#E(\mathbb{Q})_{\textrm{tors}}$ for some integer $t$. Suppose $f$ is a newform of level $N'$ and $q$ be a prime with $q \nmid N'$, $q^2 \nmid N$. Also let 
$$
S_q = \left\{a \in \mathbb{Z} : -2\sqrt{q} \leq a \leq 2\sqrt{q},\, a \equiv q+1 \,\,(\mathrm{mod}\ t)  \right\}. 
$$
Let $c_q$ be the $q$-th coefficient of $f$ and define 
$$
B_{q}^{'} (f) := \mathrm{Norm}_{K/\mathbb{Q}} \left(   (q+1)^2 - {c_q}^2\right) \prod_{a\in S_q} \mathrm{Norm}_{K/\mathbb{Q}} (a - c_q)
$$
and 
\begin{equation*}
B_q(f) = \begin{cases}
q\cdot B_{q}^{'}(f) \quad \mathrm{if }\,\, f\,\, \mathrm{is}\,\, \mathrm{irrational}, \\
B_{q}^{'}(f) \quad\, \quad \mathrm{if }\,\, f\,\, \mathrm{is}\,\, \mathrm{rational}.
\end{cases}
\end{equation*}
If $E\sim_p f$, then $p | B_q(f)$. 

\end{proposition}

\subsection{Descent}
We use the following well known method to eliminate remaining cases left after applying the methods stated above.
\par
Consider the equation in integers $R, X , S, Y, T,$ 
\begin{equation} \label{descent}
RY^p - SX^{2p} = T
\end{equation}
with $R, S, T$ pairwise coprime integers.

\smallskip
For a prime $q$, we define
$$
S^{'}=\prod_{\mathrm{ord}_{q}(S)\;\mathrm{is}\;\mathrm{odd}}q. 
$$  

Then $SS^{'}=v^{2}$ for some integer $v$. Take $RS^{'}=u$ and $TS^{'}=mn^{2}$ for some integers $u,m$ and $n$ with $m$ squarefree. Substituting these values in the equation \eqref{descent}, we have 
\begin{equation*}
(vX^{p}+n\sqrt{-m})(vX^{p}-n\sqrt{-m})=uY^{p}.
\end{equation*}
Let $K=\mathbb{Q}(\sqrt{-m})$ and $\mathcal{O}$ be its ring of integers. Let $P$ be the set of prime ideals of $\mathcal{O}$ which divide $u$ and $2n\sqrt{-m}$. The $p$-Selmer group is given by
$$
K(P,p)=\{\epsilon\in K^{*}/K^{*p}:\mathrm{ord}_{\mathcal{P}}(\epsilon) \equiv 0\ (\mathrm{mod}\ p)\;\; \mathrm{for}\;\; \mathcal{P} \not \in P\}
$$
 and this is a $\mathbb{F}_{p}$ vector space of finite dimension.
Let 
$$
\Theta=\{\epsilon \in K(P,p): \mathrm{Norm}(\epsilon)/u\in \mathbb{Q}^{*p}\}.
$$
Now it is easy to see that 
\begin{equation}\label{equn26}
vX^{p}+n\sqrt{-m}=\epsilon Z^{p},
\end{equation}
where $\epsilon \in \Theta$ and $Z\in K^{*}$.
\begin{lemma}[\cite{BPS}, Lemma 9.1]\label{lemma6}
Let $\mathfrak{q}$ be a prime ideal of $K$. Suppose one of the following holds:
\begin{enumerate}
\item $\mathrm{ord}_{\mathfrak{q}}(v), \mathrm{ord}_{\mathfrak{q}}(n\sqrt{-m}), \mathrm{ord}_{\mathfrak{q}}(\epsilon)$ are pairwise distinct modulo $p$;
\item $\mathrm{ord}_{\mathfrak{q}}(2v), \mathrm{ord}_{\mathfrak{q}}(\epsilon), \mathrm{ord}_{\mathfrak{q}}(\bar{\epsilon})$ are pairwise distinct modulo $p$;
\item $\mathrm{ord}_{\mathfrak{q}}(2n\sqrt{-m}), \mathrm{ord}_{\mathfrak{q}}(\epsilon), \mathrm{ord}_{\mathfrak{q}}(\bar{\epsilon})$ are pairwise distinct modulo $p$.
\end{enumerate}
Then there is no $X \in \mathbb{Z}$ and $Z \in K$ satisfying the equation \eqref{equn26}.
\end{lemma}

\begin{lemma}[\cite{BPS}, Lemma 9.2]\label{lemma7}
Let $q=2kp+1$ be a prime. Suppose $q\mathcal{O}=\mathfrak{q}_{1}\mathfrak{q}_{2}$ where $\mathfrak{q}_{1},\mathfrak{q}_{2}$ are distinct prime ideals in $\mathcal{O}$, such that $\mathrm{ord}_{\mathfrak{q}_{j}}(\epsilon)=0$ for $j=1,2$. Let
$$
\chi(p,q)=\{\eta^{p}:\eta \in \mathbb{F}_{q}\}.
$$
Let 
$$
C(p,q)=\{\zeta \in \chi(p,q): ((v\zeta+n\sqrt{-m})/\epsilon)^{2k} \equiv 0\;\mathrm{or}\; 1\ (\mathrm{mod}\ \mathfrak{q}_{j})  \mathrm{for}\;\;j=1,2 \}.
$$
Suppose $C(p,q)=\phi$. Then there is no $X \in \mathbb{Z}$ and $Z \in K$ satisfying the equation \eqref{equn26}.
\end{lemma}

\begin{lemma} [\cite{BPS}, Lemma 9.3]\label{lemma7a}
Suppose
\begin{enumerate}
\item $\mathrm{ord}_{\mathfrak{q}}(n\sqrt{-m})< p$ for all prime ideals $\mathfrak{q}$ of $\mathcal{O};$
\item the polynomial $U^p +(\rho-U)^p -2$ has no roots in $\mathcal{O}$ for $\rho= 1,-1,-2;$
\item the only root of the polynomial $U^{p}+(2-U)^p -2$ in $\mathcal{O}$ is $U=1$.
\end{enumerate}
Then, for $\epsilon= n\sqrt{-m},$ the only solution to equation \eqref{equn26} with $X\in \mathbb{Z}$ and $Z\in K$ is $X=0$ and $Z=1$.
\end{lemma}

\section{Proof of Theorem \ref{thm1} for $p \geq 5$}


In this section, we use lower bounds for linear forms in logarithms to bound the exponent $p$ appearing in \eqref{equn8}. We use a special case of Corollary 2 of Laurent \cite{ML}.

\begin{lemma}\label{lemma1}
Let $p> 19$. Consider
\begin{equation}\label{equn10}
 \alpha_1 = \beta/\alpha^2 \quad \textrm{and} \quad \alpha_2 = {z_{1}^2/z_2}\  (\neq 1) 
\end{equation}
 with $|z_1| \geq 2$ and $z_2 \geq 2$.

Then $\alpha_1$ and $\alpha_2$ are positive and multiplicatively independent. Moreover, if we write 
\begin{equation}\label{equn11}
\Lambda = \log \alpha_1 - p \log \alpha_2, 
\end{equation}
then
\begin{equation}\label{equn12}
0 <  \Lambda < \frac{3 d (d+1)}{\alpha^2 z_{1}^{2p}}.
\end{equation}
\end{lemma}
\begin{proof}
One can see that $\alpha_1$ and $\alpha_2$ are positive as $\beta > 0$ and $z_2 > 0$. 
From the equations \eqref{equn8},\eqref{equn10} and \eqref{equn11}, we have 
$$ 
e^{\Lambda} - 1  = \frac{\beta z_{2}^p}{\alpha^2 z_{1}^{2p}} - 1 = \frac{3 d (d+1) }{\alpha^2 z_{1}^{2p}}>0.
$$
Therefore $0<\Lambda < \frac{3 d (d+1)}{\alpha^2 z_{1}^{2p}}$ since $e^x -1>  x$ for any positive real number $x$.

Now we want to prove that $\alpha_1$ and $\alpha_2$ are multiplicatively independent. On contrary, let us suppose that $\alpha_1$ and $\alpha_2$ are not multiplicatively independent i.e., there exist co-prime positive integers $a$ and $b$ such that $\alpha_{1}^a = \alpha_{2}^b$. 
Clearly $\alpha_1 \neq 1$.\\
Then $a\ \textrm{ord}_l(\alpha_1) = b\ \textrm{ord}_l(\alpha_2)$for all prime $l.$ Hence $b| \textrm{ord}_l(\alpha_1).$

\smallskip
Let $g= \gcd \{ \textrm{ord}_l(\alpha_1):\ l\ \textrm{is prime}\}.$
From \eqref{equn11}, we have
\begin{equation} \label{equn12c}
\Lambda = \log \alpha_1 \left( 1 - p \frac{\log \alpha_2}{\log \alpha_1}  \right) =|\log\alpha_1|  \left| 1-p\frac{a}{b} \right|.
\end{equation}

Hence from \eqref{equn12} and \eqref{equn12c}, we have 
\begin{equation}\label{equn13}
0<\frac{1}{g}\leq \left| 1-p\frac{a}{b} \right| < \frac{3 d (d+1)}{|\log \alpha_1 |  \alpha^2 z_{1}^{2p}},
\end{equation}
as $b\mid g$.

\smallskip
Since $|z_1| \geq 2$, from the equation \eqref{equn13} , it follows that
\begin{equation*}\label{equn15}
4^p \leq z_{1}^{2p} < \frac{3 d (d+1 ) g}{|\log \alpha_1| {\alpha}^2}.
\end{equation*}
Therefore,
\begin{equation}\label{equn15a}
p \leq \log \left( \frac{3 d (d+1 ) g}{|\log \alpha_1| {\alpha}^2} \right) /{\log 4}.
\end{equation}
We wrote a {\it{Magma}} script to compute the bound on $p$ for $1 \leq d \leq 50$. The maximum possible value for the R.H.S of \eqref{equn15a} is $18.11$ corresponding to $d=48$ and $(\alpha, \beta) = (1/7056, 7056)$, which is not possible as $p >19$. This completes the proof of lemma.
\end{proof}

\begin{lemma}\label{lemma2}
Let $p> 1000$. Consider
\begin{equation*}
 \alpha_1 = \beta/\alpha^2 \quad \textrm{and} \quad \alpha_2 = {z_{1}^2/z_2}\  (\neq 1) 
\end{equation*}
 with $|z_1| \geq 2$ and $z_2 \geq 2$. Then we have
$$
\frac{\log z_2}{\log z_{1}^2} \leq 1.01 . 
$$
\end{lemma}
\begin{proof}
From the equations \eqref{equn10}, \eqref{equn11} and \eqref{equn12}, we have 
\begin{equation}
 \log \alpha_1 -  p( \log z_{1}^2 - \log z_2) < \frac{3 d (d+1)}{\alpha^2 4^p}.
 \end{equation}
 Hence
 \begin{equation}
\begin{split}
 \frac{\log z_2}{\log z_{1}^2} &< 1 + \frac{1}{p \log z_{1}^2 }\left(   \frac{3 d (d+1)}{\alpha^2 4^p} - \log \alpha_1 \right)  \\
& \leq  1 + \frac{1}{p \log z_{1}^2 }\left(   \frac{3 d (d+1)}{\alpha^2 4^p} + | \log \alpha_1| \right)  \\
& \leq  1 + \frac{1}{1000 \log 4 }\left(   \frac{3 d (d+1)}{\alpha^2 4^{1000}} + | \log \alpha_1| \right),
\end{split}
\end{equation}
where $p > 1000$ and $z_2 \geq 2$. We write a {\it{Magma}} script to find the maximum possible value of the right-hand side which is 1.02, corresponding to $d=50$ and $(\alpha, \beta) = ( 7650, 1/7650)$. This completes the proof. 
\end{proof}

Now we are ready to apply Proposition \ref{Linearform} to find a upper bound for the exponent $p$.
 
\begin{lemma}\label{lemma3}
 Let $(z_1, z_2) $ be an integral solution of the equation \eqref{equn8} with $|z_1|, z_2 \geq 2$ and $z_{1}^2 \neq z_2$, where $1 \leq d \leq 50$ and $(\alpha, \beta) \in S_d$. Then we have $p < 4 \times 10^4$.
\end{lemma}
\begin{proof}
Let $A_1 = \max \{H(\alpha_1), e \}$, where $H(a/b) = \max \{ |a|, |b| \}$ for $\alpha_1=\frac{a}{b}$. Let $A_2= \max \{{z_1}^2,z_2\}.$ From Lemma \ref{lemma1}, it is clear that the hypothesis of Theorem \ref{Linearform} is satisfied for our choices of $\alpha_1, \alpha_2, A_1, A_2$ with $D=1$. 
Let 
$$
b'= \frac{1}{\log A_2} + \frac{p}{\log A_1}.
$$
As $p>1000$, we have $b' > \frac{1000}{\log A_1}$. For $1 \leq d \leq 50$ and $(\alpha, \beta) \in S_d $, the lower bound for $1000/{\log A_1} $  is $37.27$ corresponding to $d=50$ and $(\alpha, \beta) = ( 7650, 1/7650)$. Now apply Theorem \ref{Linearform}, we have 
\begin{equation}
\log |\Lambda |  \geq -25.2  \left(\max \{\log b' + 0.38, 10, 1  \}    \right)^2 \log A_1 \log A_2.
\end{equation}
Further, this implies
\begin{equation}
\begin{split}
- \log \Lambda & \leq 25.2 \log A_1 \log A_2 (\log b')^2 \\
& \leq 25.2 \log A_1 \log A_2 \log^2 \left( \frac{p}{\log A_1} + \frac{1}{\log 4}   \right).
\end{split}
\end{equation}
From equation \eqref{equn12}, we have 
\begin{equation}
  p \log z_{1}^2 - \log \left(\frac{3 d (d+1) }{\alpha^2}   \right) < 25.2 \log A_1 \log A_2 \log^2 \left( \frac{p}{\log A_1} + \frac{1}{\log 4}   \right).  
\end{equation}
Finally, we conclude 
\begin{equation}
  p < \frac{1}{\log z_{1}^2} \left\{ \log \left(\frac{3 d (d+1) }{\alpha^2}   \right) + 25.2 \log A_1 \log A_2 \log^2 \left( \frac{p}{\log A_1} + \frac{1}{\log 4}   \right)  \right\}.
\end{equation}
As $|z_1| \geq 2$, from  Lemma \ref{lemma2} we  have 
$$
p < \frac{1}{\log 4} \left\{ \log \left(\frac{3 d (d+1) }{\alpha^2}   \right) + 26 \log A_1 \log^2 \left( \frac{p}{\log A_1} + \frac{1}{\log 4}   \right)  \right\}.
$$
We write a {\it Magma} script to obtain $p< 4\times 10^4$. This completes the proof of the lemma. 
\end{proof}
Let $z_1$ and $z_2$ be integral solutions of \eqref{equn7}.
Then by Lemmas \ref{lemma1}, \ref{lemma2} and \ref{lemma3}, we found
$$p<4\times 10^4,\mbox{ for } |z_1| \geq 2 \mbox{ and } z_2 \geq 2  \mbox{ with } z_{1}^2 \neq z_2$$
When $z_{1}^2 = z_2$, we determine all the possible solutions for $1 \leq d \leq 50$ and these solutions $(z_1,z_2)$ are not satisfying the equation \eqref{equn4}. Similarly, if $z_1 \in \{ -1, 0, 1 \}$ or $z_2 = 1$, we determine all the possible solutions for $1 \leq d \leq 50$ and we observe that $ (20,-15,6,5), (27,26,6,7)$ are the only integral solutions for $(d,x,z,p)$ satisfying the equation \eqref{equn4}. Hence we conclude that the equation \eqref{equn4} has no integral solutions for $p > 4\times 10^4$ . 

\smallskip

For $1\leq d \leq 50, (\alpha,\beta)\in S_{d}$ and $5\leq p \leq 4 \times 10^4$, we wrote a {\it Magma} script with $k\leq 765$, that searches for a prime 
$q$ satisfying $q=2kp+1\nmid r$ such that $B(p,q) = \phi$.\\
We note that if there exist such a prime $q$ with $B(p,q) = \phi$, then by Lemma \ref{lemma4} the equation \eqref{equn8} has no solution for exponent $p$. This criterion fails when $\beta = 3d(d+1)$ (equivalently $r = t$) for which we have the trivial solution $(z_1,z_2) = (0,1)$. In addition, for $\beta \neq 3d(d+1)$ (equivalently $r = t$) we found 1716 quintuples $(d,p,r,s,t)$ which fails to satisfy this criterion.
\par 

Now, to complete the proof of Theorem \ref{thm1} for $p\geq 5$, we are remaining with the following cases.

\begin{enumerate}
\item
$r = t$ and $p<5 \times 10^4 $
\item
$r \neq t$ and $p<5 \times 10^4 $ consisting 1716 quintuples $(d,p,r,s,t)$.
\end{enumerate}

 To solve the equation \eqref{equn9} for $r = t$ and $p<5 \times 10^4$ we want to apply modular method. Here we use the recipes of Kraus \cite{AK} due to Wiles \cite{AW}, Ribet \cite{KR} and Mazur \cite{BM}.
\par 

In the case $r=t$, the equation \eqref{equn8} has a solution $(z_{1}, z_{2})=(0,1)$. In fact, we want to show that $(z_{1}, z_{2})=(0,1)$ is the only solution.
\par 
Since $r=t$, we have $\alpha=1/3d(d+1)$ and thus the equation \eqref{equn8} will reduce to 
\begin{equation}\label{equn20}
z_{2}^{p}-\frac{1}{(3d(d+1))^3}z_{1}^{2p}=1.
\end{equation}
Let $R=\mathrm{Rad}~(3d(d+1))$. Since  $z_{1}$ and $z_{2}$ are integers, we have $R\mid z_{1}$. Hence $z_{1}=Rz_{3}$ for some integer $z_3$. Then from the equation \eqref{equn20}, we have 
\begin{equation*}
z_{2}^{p}-\frac{R^{2p}}{(3d(d+1))^3}z_{3}^{2p}=1.
\end{equation*}
Take $T=\frac{R^{2p}}{(3d(d+1))^3}$ then the above equation becomes
\begin{equation}\label{equn21}
z_{2}^{p}-T z_{3}^{2p}=1.
\end{equation}
It is easy to see that $\mathrm{Rad}(T)=R$. Further we assume that 
\begin{equation}\label{equn22}
2p> 3\cdot \mathrm{ord}_{q}(3d(d+1))
\end{equation} 
for all odd primes $q$. We want to show that $z_{1} = 0$ for the equation \eqref{equn20}. On contrary, let us assume that $z_{1}\neq 0$, which implies $z_{3}\neq 0$. Also $z_{2}\neq 0$.
The equation \eqref{equn21} can be written in the following form 
$$ A x^p + B y^p + C z^p = 0,  $$
where $A= -1, B = -T, C=1, x= 1, y=z_{3}^2, z= z_2$ and also 
$$ Ax^p \equiv -1\ (\textrm{mod}\ 4), By^p \equiv 0\ (\textrm{mod}\ 2).$$

 Now we associate a solution $(z_{2},z_{3})$ to the Frey Curve 
\begin{equation}\label{equn23}
E:\quad Y^2=X(X+1)(X-Tz_{3}^{2p}).
\end{equation}
The Weierstrass model given in \eqref{equn23} is smooth as $z_{2}z_{3}\neq 0$. Let $E \sim_{p} f$, where $f$ is a weight $2$ newform of level $N_p$ with $N_p$ is defined as follows:

\begin{equation}\label{equn24}
        N_p = \begin{cases}                        
        R \quad \quad \text{if  $\textrm{ord}_2 (T) =0 \,\, \textrm{or} \,\,\geq 5, $ }\\
            \frac{R}{2} \quad\quad \text{if  $\textrm{ord}_2 (d (d+1)) =2 \,\, \textrm{and} \,\,p=5, $ }    \\
               R \quad\quad \text{if  $\textrm{ord}_2 (d (d+1)) =3, \,\, \,\,p=5 \,\,\textrm{and}~~z_{3}~~ \textrm{even}, $ }\\    
                  R \quad \quad \text{if  $\textrm{ord}_2 (d (d+1)) =4, \,\,p=7\,\,\textrm{and}~~z_{3}~~ \textrm{even}, $ }\\    
                  2^2 R \quad \text{if  $\textrm{ord}_2 (d (d+1)) =4, \,\,p=7\,\,\textrm{and}~~z_{3}~~ \textrm{odd}, $ }\\    
                  2^4 R \quad \text{if  $\textrm{ord}_2 (d (d+1)) =3, \,\,p=5\,\,\textrm{and}~~z_{3}~~ \textrm{odd}. $ }\\    
                    \end{cases}
\end{equation}

Suppose $f$ is rational and hence we  get an elliptic curve $F$ of conductor $N_{p}$. Now we choose a prime $q=2kp+1 $ such that $q\nmid N_{p}$ and $E$ has multiplicative reduction at $q$. Then by Proposition \ref{prop3}, $q+1 \equiv \pm a_{q}(F)\ (\textrm{mod}\ p)$ and this will imply $4 \equiv (a_{q}(F))^2\ (\textrm{mod}\ p)$ as $q \equiv 1\ (\textrm{mod}\ p)$. 

\par 
Suppose that $f$ is irrational. Since $c_q\not \in \mathbb{Q}$ for infinitely many coefficients of $f$, we have  $B_{q}(f)\neq 0$ for infinitely many primes $q$. Then Proposition \ref{prop4} allows us to bound $p$. In fact, this bound is very small. Here we improve this bound by choosing a set of primes $\mathcal{P}=\{q_{1},\ldots,q_{n}\}$ such that $q_{i}\nmid N_p$ for all $i$ and $B_{\mathcal{P}}(f)=\gcd(B_{q}(f): q\in \mathcal{P})$. Thus, if $E\sim_p f$ then $p\mid B_{\mathcal{P}}(f)$.
\par 

\par 

From the above observations, the following lemma which is a variant of Lemma 7.1 in \cite{BPS}, is very helpful to eliminate newforms of level $N_{p}$. Condition (1) in Lemma \ref{lemma5} is equivalent to say that $E$ has multiplicative reduction at $q$.

\begin{lemma}\label{lemma5}
Let $1\leq d\leq 50$. Also let $p\geq 5$ be a prime which satisfies the inequality  \eqref{equn22} for all primes $q$. Let $N_{p}$ be given in \eqref{equn24}. Suppose for each irrational newform $f$ of weight $2$ and level $N_{p}$ there is a set of primes $\mathcal{P}$ not dividing $N_{p}$ such that $p\nmid B_{\mathcal{P}}(f)$. Suppose for every elliptic curve $F$ of conductor $N_{p}$ there is a prime $q=2kp+1, q\nmid N_{p}$, such that 
\begin{enumerate}
\item $ B(p,q)=\{\bar{0}\}$, where $B(p,q)$ is in statement of Lemma \ref{lemma4};
\item $ p\nmid (a_{q}(F)^2-4)$. 
\end{enumerate}
Then the equation \eqref{equn4} has only one solution with $(\alpha,\beta)=\left(\frac{1}{3d(d+1)},3d(d+1)\right)$ satisfying $x=-(d+1)$.
\end{lemma}

Now we write a {\it Magma} script for each $1\leq d \leq 50$ which computes the newforms of weight $2$ and level $N_{p}$. Here we assume that $\mathcal{P}$ is the set of primes $<100$ that do not divide $N_{p}$. Then for each irrational newform we compute $B_{\mathcal{P}}(f)$.
\par
For every prime $5\leq p < 4 \times 10^4$ that do not divide $B_{\mathcal{P}}(f)$, satisfies the inequality \eqref{equn22} and for every isogeny class of elliptic curves $F$ of conductor $N_{p}$, we search for the primes $q=2kp+1, q\nmid N_{p}$ with $k<765$ such that condition (1) and (2) of Lemma \ref{lemma5} hold.
\par
If we find such a prime then the equation \eqref{equn4} has no solution with $r=t$. The criterion holds for all values of $p$ except for few small values of $p$.  When $N_{p} = R$, there are $55$ cases  where either $p$ does not satisfy the inequality \eqref{equn22}, or it divides $B_{\mathcal{P}}(f)$ for some irrational newform  $f$, or $q$ do not satisfy condition (1) and (2) of Lemma \ref{lemma5}.
\par
For other special cases of $N_{p}$ we are remaining $3$ equations, which do not satisfy the above conditions. The largest value of $p$ among the $58$ quintuples is $p = 19$  with $d = 37,
\alpha = 1/4218, \beta = 4218, r = t = 75044648232, s = 1$.

Now we have total $1716+55+3=1774$ remaining equations, which can not be eliminated by Lemma \eqref{lemma4} and modular approach. These equations are of the form \eqref{equn9} with $r,s$ and $t$ positive integers and $\gcd(r,s,t)=1$. There is a possibility that $r,s$ and $t$ may not be pairwise coprime. We apply the procedure mentioned in [\cite{BPS},section 9.1] which is nothing but a repeatative way of clearing out the common factor to get an equation of the form 
\begin{equation}\label{equn25}
RY^{p}-SX^{2p}=T
\end{equation}  
where $R,S, T$ are pairwise coprime and $X, Y$ are divisors of $z_1,z_2$ respectively.
\par
If there exist a solution for the equation\eqref{equn25}, then $-ST$ is a square modulo $q$ for any odd prime $q\nmid R$. Also we check for local solubility at the primes dividing $R, S,T$ , and the primes $q\leq 19$. Applying these above tests, we are remaining with $175$ equations after elimination. For these remaining equations we apply descent.
\par
By applying Lemma \ref{lemma6} and \ref{lemma7} to the remaining equations, which were left after local solubility, we eliminate $\epsilon \in \Theta$. But we know that if $r=t$ then the equation \eqref{equn9} has a solution, i.e.,  $(z_{1},z_{2})=(0,1)$. For $r=t$, the reduction process leads to $R=T=1$. Thus the solution $(z_{1},z_{2})=(0,1)$ in \eqref{equn9} corresponds to $(X, Y)=(0,1)$ in \eqref{equn25}. Also $n\sqrt{-m}(K^*)^p\in \Theta$. Hence using Lemma \ref{lemma6} and \ref{lemma7}, we eliminate all $\epsilon$ except the case $\epsilon=n\sqrt{-m}$ as the equation \eqref{equn26} has a solution namely, $(X,Z) = (0, 1)$.
\par
For the case $\epsilon=n\sqrt{-m}$, the equation \eqref{equn25} has only one solution $(X,Y)=(0,1)$ by Lemma \ref{lemma7a}. If $X=0$ then $z_{1}=0$ and hence, $x=-(d+1)$. If Lemma \ref{lemma6}, \ref{lemma7} and Lemma \ref{lemma7a} allow us to conclude $X=0$, then we can eliminate $(r,s,t)$ as we can consider $x \neq -(d + 1)$. We write a {\it Magma} script for above procedure and we eliminate $164$ equations. Now we have to solve only $11$ remaining equations by Thue approach. By writting $V=Y^2$ in \eqref{equn25}, we obtain the Thue equation 

\begin{equation}
RY^{p}-SV^{p}=T.
\end{equation}     

Using Thue equation solver in {\it Magma}, we solve the remaining equations. Finally we have the follwing solutions.
\begin{align}\label{equn27}
\begin{split}
&27^3-28^3-\cdots-80^3+81^3=6^7,\\
&(-2)^3-(-1)^3+\cdots-5^3+6^3=2^7,\\
&(-12)^3-(-11)^3+\cdots-17^3+18^3=3^7,\\
&(-14)^3-(-13)^3+\cdots-25^3+26^3=6^5.
\end{split}
\end{align}
This concludes the proof of Theorem \ref{thm1} for $p\geq 5$.
\section{Proof of Theorem \ref{thm1} for $p =2$}

Putting $x+d+1 = u$ and $p = 2$ in the equation \eqref{equn4}, we have 
\begin{equation}\label{equn5}
z^2 =  u^3 + 3 d (d+1) u.
\end{equation}
This represents an elliptic curve, as $u^3 + 3d(d+1)u=0$ has no multiple roots. For $1\leq d\leq 50$, we obtain the integral solutions of the equation \eqref{equn5} by {\it Magma}. These solutions give rise to all the integral solutions of \eqref{equn2} and those are given explicitly in Table \ref{tab1}. 

\section{Proof of Theorem \ref{thm1} for $p=3$}
In this case the required equation is 
\begin{equation}\label{equn6}
z^3 =  u^3 + 3 d (d+1) u.
\end{equation}
One can easily see that the divisors of $3d(d+1)$ divide $\gcd(u, u^2 + 3 d (d+1))$. So if $\alpha=\gcd(u, u^2 + 3 d (d+1))$, then 
\begin{equation}\label{equn6a}
u=\alpha u_{1} \;\; \mathrm{and} \; u^2 + 3 d (d+1)=\alpha u_{2},
\end{equation}
where $\gcd(u_{1},u_{2})=1$.
 For non-negative integers $l, m \equiv 0\ (\mathrm{mod}\ 3)$, we can write
 \begin{align}\label{equn6b}
 \begin{split}
& u_{1}=2^{l}\cdot u_{3}\;\; \mathrm{and} \; u_{2}= 3^{m}\cdot u_{4}, \\
\mathrm{or}\;\; &u_{1}=3^{l}\cdot u_{3} \;\; \mathrm{and} \; u_{2}= 2^{m}\cdot u_{4},\\ 
 \mathrm{or }\;\;& u_{1}=2^{l}\cdot 3^{m}\cdot u_{3}\;\; \mathrm{and}\; \; u_{2}= u_{4},\\
\mathrm{or}\;\; &u_{1}=u_{3} \;\; \mathrm{and} \; u_{2}= 2^{l}\cdot 3^{m} \cdot u_{4},
\end{split}
\end{align} 
where $u_3$ and $u_4$ are integers with $\gcd(u_{3},u_{4})=1$.
 
Also write $\alpha=2^{\delta_2}\cdot 3^{\delta_3} \cdot \alpha_{1}$ for some integer $\alpha_{1}$
with $\delta_2:=\mathrm{ord}_{2}(\alpha)\;\; \mathrm{and} \;\delta_3:=\mathrm{ord}_{3}(\alpha)$.
As $\alpha u_{1}\cdot \alpha u_{2}=z^3$, we have $2^{2\delta_{2}+l}3^{2\delta_{3}+m}\alpha_{1}^{2}u_{3}u_{4}=z^{3}$ for $u_{1}=2^{l}\cdot u_{3}, u_{2}= 3^{m}\cdot u_{4}$ and this will imply $\alpha_{1}^{2}u_{3}u_{4}=z_{1}^{3}$ for some integer $z_1$. Since $1\leq d \leq 50$, for any prime $q\mid \alpha_{1}, \mathrm{ord}_{q}(\alpha_{1})=2$. Therefore, we can conclude that, if $\alpha_{1}^{2}\mid z_{1}^{3}$ then $\alpha_{1} \mid z_{1}$. Write $z_{1}=\alpha_{1}\cdot z_{2}$ for some integer $z_2$, hence we have $u_{3}\cdot u_{4}=\alpha_{1}z_{2}^{3}$. Since $\gcd(u_{3}, u_{4})=1$, we can write 
\begin{equation}\label{equn6c}
u_{3}=\alpha_{2}\cdot z_{3}^{3}\; \mathrm{and}\;u_{4}=\alpha_{3}\cdot z_{4}^{3},
\end{equation}
 for some integers $\alpha_{2}, \alpha_{3}, z_{3},z_{4}$ with $\alpha_{2}\alpha_{3}=\alpha_{1}$ and $z_{3}z_{4}=z_{2}$. Rewriting the equation \eqref{equn6a}, we have  
\begin{equation}\label{equn6d}
\alpha \cdot u_{2} -\alpha^2\cdot u_{1}^{2}=3d(d+1).
\end{equation}
Now from equations \eqref{equn6b},\eqref{equn6c} and \eqref{equn6d},  we will have a set of Thue equations as follows:
\begin{align}
 \begin{split}
& \alpha \cdot \alpha_{3}\cdot  3^{m} \cdot z_{4}^3 -\alpha^2 \cdot\alpha_{2}^2 \cdot  2^{2l} \cdot (z_{3}^{2})^3=3d(d+1),\\
\mathrm{or}\;\;& \alpha \cdot \alpha_{3}\cdot  2^{m} \cdot z_{4}^3 -\alpha^2 \cdot\alpha_{2}^2 \cdot  3^{2l} \cdot (z_{3}^{2})^3=3d(d+1),\\
\mathrm{or}\;\;& \alpha \cdot \alpha_{3}\cdot z_{4}^3 -\alpha^2 \cdot\alpha_{2}^2 \cdot  2^{2l} \cdot 3^{2m} \cdot (z_{3}^{2})^3=3d(d+1),\\
\mathrm{or}\;\;& \alpha \cdot \alpha_{3}\cdot 2^{2l} \cdot 3^{2m} \cdot z_{4}^3 -\alpha^2 \cdot\alpha_{2}^2 \cdot  (z_{3}^{2})^3=3d(d+1).
 \end{split}
\end{align}
 
Now, for $1\leq d \leq 50$ and $l, m \equiv 0\ (\mathrm{mod}\ 3)$ we have written a {\it Magma} script to solve these four Thue equations. The theory about solving these Thue equations is discussed in \cite{NPS}. Using backward calculations from these solutions we find all solutions for the equation \eqref{equn6} and 
these are given explicitly in Table \ref{tab1}. 

\begin{table}[h!]
\begin{center}
 \begin{tabular}{||c| p{12cm}||} 
 \hline
 $d$ & $(x,z,p)$  \\ [0.5ex] 
 \hline\hline
 $d$ & $(-d-1, 0, p)$\\
 \hline
 $2$ & $(0, \pm 9, 2), (3,\pm 18,2), (69,\pm 612, 2)$  \\ 
  \hline
 4  & $(-3, 2, 7), (1, \pm 24, 2), (5,\pm 40, 2), (235,\pm 3720, 2)$  \\
 \hline
 5  & $(34,\pm 260, 2)$ \\
 \hline
 6 & $(0,\pm 35, 2), (11,\pm 90, 2)$   \\ 
 \hline
 7  & $(-7, \pm 13, 2),(160,\pm 2184)$ \\
 \hline
  8  & $(16,\pm 145, 2)$ \\
 \hline
  11  & $(36,\pm 360, 2)$ \\
 \hline
  12  & $(-9,\pm 44, 2),(0,\pm 91, 2), (23, \pm 252, 2), (104, \pm 1287,2),$\\
  & $(195,\pm 3016, 2)$ \\
  \hline
  15  & $(-13, 3, 7)$ \\
 \hline
  16  & $(83,\pm 1040, 2)$ \\
 \hline
  19  & $(-16,\pm 68, 2),(-14, \pm 84,2), (34, \pm 468,2),$\\
  &$(170, \pm 2660,2),(265, \pm 4845,2), (5746,\pm 437844,2)$ \\
 \hline
  20  & $(-15, 6, 5), (0,\pm 189,2), (39, \pm 540, 2)$ \\
 \hline
  26 & $(-39, -30, 3), (-36, -27, 3), (-18, 27, 3), (-15, 30, 3)$\\
 \hline
  27  & $(-10,\pm 216, 2), (-46, -36, 3), (-34, -24, 3), (-22, 24, 3),$\\ 
  &$(-10, 36, 3), (84, \pm 1288, 2), (98, \pm 1512, 2),$\\
  & $(39734, \pm 7928712, 2), (26,6,7)$ \\
 \hline
  28  & $(13,\pm 420, 2), (29, \pm 580, 2)$ \\
 \hline
  29  & $(-24,\pm 126, 2), (405, \pm 9135, 2)$ \\
 \hline
  30  & $(-21,\pm 170, 2),(0,\pm 341, 2), (59, \pm 990, 20), $\\
  & $(248, \pm 4743, 2),(1179, \pm 42130, 2), (5208, \pm 379223, 2)$ \\
 \hline
  32 & $(-24,\pm171, 2), (319, \pm 6688, 2)$ \\
 \hline
  34  & $(16,\pm 561, 2), (35, \pm 770, 2), (14245, \pm 1706460, 2)$ \\
 \hline
  36   & $(-91, -72, 3), (-39, -20, 3), (-35, 20, 3), (17,72, 3)$ \\
 \hline
  38  & $(2811,\pm 152190, 2)$ \\
 \hline
  39 & $(-31,\pm 207, 2), (81, \pm 1529, 2), (480, \pm 11960, 2)$ \\
 \hline
  42  & $(-124, -99, 3), (0,\pm 559, 2), (83, \pm 1638, 2), (38, 99, 3)$ \\
 \hline
  45  & $(8,\pm 702, 2), (69, \pm 1495, 2), (440, \pm 10854, 2)$ \\
 \hline
  47  & $(-36,\pm 288, 2), (516, \pm 13536, 2)$ \\
 \hline
  49  & $(230,\pm 4900, 2), (-95, -75, 3), (-5, 75, 3)$ \\ [1ex] 
 \hline
\end{tabular}
\caption{The integral solutions to equation \eqref{equn2} for $r =2d+1$ with $1\leq d \leq 50$ and $p$ is prime. }\label{tab1}
\end{center}
\end{table}

\begin{remark}
For $r = 2d$ the equation \eqref{equn2} becomes
\begin{equation*}\label{equn3a}
d[3x^2 + 3(2d+1)x + d(4d+3)] = (-z)^p.
\end{equation*}
Since the polynomial $3dx^2 + 3d(2d+1)x + d^2(4d+3)$ is an irreducible polynomial over $\mathbb{Q}$ for $1 \leq d \leq 50$, by Theorem 12.11.2 in \cite[p. 437]{HC}, we conclude that the equation \eqref{equn2} has finitely many solutions for even $r$.

\end{remark}

\section{concluding remark}
One can view alternating sum of consecutive cubes of even length as alternating sum of consecutive cubes of odd length by using symmetry around $0$. But in that case, we have to deal with alternating sum of consecutive cubes of higher length. So if we get $d > 50$, then we can not conclude anything about getting perfect powers in the alternating sum of even length. However, if we get $d \leq 50$, then we can find out perfect powers in the alternating sum of even length. For example, the last three equations in \eqref{equn27} can be re-written as
\begin{align}\label{equn28}
\begin{split}
&3^3-4^3+5^3-6^3=(-2)^7,\\
&13^3-14^3+15^3-16^3+17^3-18^3=(-3)^7,\\
&15^3-16^3+\cdots+25^3-26^3=(-6)^5.
\end{split}
\end{align}

Hence we see that $(2,2,-2,7), (3,12,-3,7), (6,14,-6,5)$ are solutions for $(d,x,z,p)$ in the equation\eqref{equn2}.
In general, when $r$ is even in the equation \eqref{equn2}, we conjecture the following. 

\begin{conj}
Let $r = 2 d $ with $ 1\leq d \leq 50$ and $p \geq 5$ be a prime number. Then the possible integer solutions of the equation \eqref{equn2} are given by
$$
(d,x,z,p) = \{(2,2,-2,7), (3,12,-3,7), (6,14,-6,5), (27,215,-9,7)\}.$$
\end{conj}

\end{document}